\documentclass[a4paper, 11pt]{article}
\usepackage{amsmath, amsthm, amssymb, fullpage, graphicx, enumerate, bm}
\usepackage[shortlabels]{enumitem}
\usepackage{mathtools}
\usepackage{graphicx}
\usepackage{marginnote}
\usepackage{xcolor}

\usepackage{centernot}

\numberwithin{equation}{section}

\theoremstyle{plain}
\newtheorem{theorem}{Theorem}[section]
\newtheorem{lemma}[theorem]{Lemma}

\theoremstyle{definition}

\theoremstyle{remark}
\newtheorem{remark}[theorem]{Remark}

\newcommand{\Ex}{\mathbb E}
\renewcommand{\Pr}{\mathbb{P}}
\newcommand{\Pf}{\mathbf{P}}

\newcommand{\N}{\ensuremath{\mathbb N}}
\newcommand{\R}{\ensuremath{\mathbb R}}

\newcommand{\Z}{\ensuremath{\mathbb Z}}

\newcommand{\Lr}[1]{Lemma~\ref{#1}}

\newcommand{\Tr}[1]{Theorem~\ref{#1}}

\begin{document}

\title{
	Gap at $1$ for the percolation threshold of Cayley graphs
}

\author{Christoforos Panagiotis\footnotemark[1]\footnote{Universit\'e de Gen\`eve, christoforos.panagiotis@unige.ch}~ and~Franco Severo\footnotemark[2]\footnote{ETH Z\"{u}rich, franco.severo@math.ethz.ch}}

\thispagestyle{empty}
\maketitle

\begin{abstract}
We prove that the set of possible values for the percolation threshold $p_c$ of Cayley graphs has a gap at $1$ in the sense that there exists $\varepsilon_0>0$ such that for every Cayley graph $G$ one either has $p_c(G)=1$ or $p_c(G) \leq 1-\varepsilon_0$. The proof builds on the new approach of \mbox{Duminil-Copin, Goswami, Raoufi, Severo \& Yadin (\emph{Duke Math. J., 2020})} to the existence of phase transition using the Gaussian free field, combined with the finitary version of Gromov's theorem on the structure of groups of polynomial growth of \mbox{Breuillard, Green \& Tao (\emph{Publ. Math. IH\'{E}S, 2012})}.
\end{abstract}

\section{Introduction}

\textbf{Motivation and main result.} In Bernoulli site percolation, each vertex of a connected, locally finite graph $G=(V,E)$ is either deleted (closed) or retained (open) independently at random with retention probability $p\in [0,1]$ to obtain a random subgraph $\omega$ of $G$. We write $\Pf_p$ for the law of $\omega$ and refer
to the connected components of $\omega$ as \emph{clusters}. A principal quantity of interest is the \emph{critical point} defined as
$$p_c=p_c(G)\coloneqq \inf\{p\in [0,1] : \omega \text{ has an infinite cluster } \Pf_p \text{-- a.s.}\}.$$

The first and most fundamental question in percolation theory is the existence of a non-trivial phase transition, which in our case corresponds to $0<p_c<1$. One can use elementary path-counting arguments \cite[Theorem 1.33]{MR1707339} to show that $p_c \geq 1/(D - 1)>0$ for every graph of maximum degree $D$. On the other hand, the complementary bound $p_c<1$ is, in general, harder to obtain. Still, in the classical setting of the hypercubic lattice $\Z^d$, a combinatorial technique, known as \emph{Peierl's argument}, allows to prove that $p_c(\Z^d)<1$ for every $d\geq 2$ \cite[Theorem 1.10]{MR1707339}. 

Since the pioneering work of Benjamini \& Schramm \cite{BenSch96}, the study of percolation beyond the classical setting of $\Z^d$, in particular on Cayley graphs and more generally, transitive and quasi-transitive graphs, has attracted a substantial amount of interest. While it is easy to prove that $p_c(G)=1$ for every graph $G$ with linear volume growth, Benjamini \& Schramm conjectured in the same paper that $p_c(G)<1$ for every quasi-transitive graph $G$ with superlinear volume growth. Several important cases of the
conjecture were initially obtained, including quasi-transitive graphs of polynomial growth -- this uses Gromov's celebrated theorem on groups of polynomial growth \cite{gromov1981groups}, see the discussion in \cite[Section 1]{DGRSY20} -- and exponential growth \cite{Lyons95},
and Cayley graphs of a finitely presented groups \cite{babson1999cut} -- see also \cite{BPP,candellero2015percolation,MP01,raoufi2017indicable,teixeira2016percolation} for other results concerning this problem. The full conjecture was finally resolved in the recent work of \mbox{Duminil-Copin, Goswami, Raoufi, Severo \& Yadin \cite{DGRSY20}}. In fact, their result applies beyond quasi-transitive graphs to any graph with isoperimetric dimension $d>4$ and bounded degree.

For arbitrary quasi-transitive graphs, $p_c(G)$ can take values arbitrarily close to $1$, but for transitive graphs one might expect that their quantitative structure theory leads to a uniformity on $p_c(G)$. Very recently, Hutchcroft \& Tointon \cite{HT21} obtained a non-trivial upper bound on $p_c(G)$ that depends only on the degree of the transitive graph $G$. Their proof is based on the approach of \cite{DGRSY20} and the recent work of Tessera \& Tointon \cite{TTfinitary} on the structure of transitive graphs, which in turn builds on the celebrated work of Breuillard, Green \& Tao \cite{BGT12} on the structure of approximate groups.

The main result of this article improves on the aforementioned result of Hutchcroft \& Tointon in the case of Cayley graphs, giving a uniform non-trivial upper bound on $p_c(G)$ that does not depend on the degree.

\begin{theorem}\label{thm:main}
There is a universal constant $\varepsilon_0>0$ such that $p_c(G)\leq 1-\varepsilon_0$ for site percolation on any Cayley graph $G$ of superlinear growth.
\end{theorem}

\begin{remark}
	By the classical inequality $p_c^{\text{bond}}\leq p_c^{\text{site}}$ relating the bond and site percolation thresholds, the uniform bound of Theorem~\ref{thm:main} also holds for bond percolation.
\end{remark}

We expect that the statement of \Tr{thm:main} can be extended to transitive graphs of superlinear growth. This would require suitable extensions of Lemmas~\ref{lemma:isop} and \ref{lemma:expansion}, which seems more challenging due to a poorer understanding of the geometry of transitive graphs.

\quad

\textbf{About the proof.} The proof of \Tr{thm:main} is split into two cases: the low-dimensional case and the high-dimensional one -- see Section~\ref{sec:preliminaries} for the definition of dimension.
In the low-dimensional case, one can use the finitary version of Gromov's theorem from \cite{BGT12} to obtain a non-trivial upper bound on $p_c$ for Cayley graphs of isoperimetric dimension $1<\mathrm{Dim}(G)\leq d$ depending only on $d$. This result is implicit in \cite{HT21}.

For the high-dimensional case, we use the method developed in \cite{DGRSY20}, which is based on a comparison between Bernoulli percolation and the Gaussian free field (GFF). We will first describe the general scheme of proof and then point out the differences as well as the main new ingredients when compared with \cite{DGRSY20}. The proof starts by observing that the site percolation model given by the excursion set $\{\varphi>-1\}\coloneqq\{x\in V:~\varphi_x>-1\}$ contains an infinite cluster almost surely -- see e.g.~\cite{BricmontLJM87} -- on any transient graph, where $\varphi = (\varphi_x)_{x\in V}$ is the GFF on $G$ -- see Section~\ref{sec:proof} for the definition. By using an interpolation argument, we then dominate the connection probabilities of this strongly correlated model by that of a standard Bernoulli percolation of parameter $p_0<1$, thus concluding that $p_c(G)\leq p_0<1$. For that purpose, we decompose the GFF into a sum $\varphi=\sum_{n\geq1} \phi^n$ of independent finite-range fields $(\phi^n)_{n\geq1}$, and then construct a hybrid model made of the superposition of a Bernoulli percolation of parameter $p$ and excursion sets of these finite-range fields above a level $\lambda$. By upper bounding the derivative of connection probabilities with respect $\lambda$ by the derivative with respect to $p$, we manage to integrate out each finite-range field while compensating it by a slight increase in $p$. The proof of this bound relies on certain \emph{local surgeries} relating different notions of pivotality events, in a similar spirit as the work of Aizenman \& Grimmett \cite{AizGri91} on essential enhancements.

In the general scheme of proof described in the previous paragraph, the large deviations of the finite-range fields -- which are related to the return probabilities $p_n(x,x)$ for the random walk on $G$ -- fight against the cost of the local surgeries -- which become worse as the degree $D$ of $G$ grows. Therefore, in order to end up with an upper bound $p_0<1$ independent of $D$, we need to prove a sufficiently good upper bound on $p_n(x,x)$ that is uniform over Cayley graphs of high enough dimension, and that furthermore becomes better as $D$ increases. Such a result is proved in Theorem~\ref{theorem:uni_pn} and is the most technical part of this article. The proof of Theorem~\ref{theorem:uni_pn} is divided into two cases: small and large $n$. For the case of large $n$, we apply a classical bound on the heat kernel in terms of the expansion profile of $G$ --  see Theorem~\ref{theorem:pn}. This reduces the problem to proving an appropriate isoperimetric inequality, which is obtained by combining a uniform lower bound on the volume of large balls provided by \cite{BGT12} and a slight improvement on a classical inequality relating growth and isoperimetry -- see Theorem~\ref{theorem:isop} and Lemma~\ref{lemma:isop} respectively. In the small $n$ case, we observe that if we assume the set of generators of a group to be minimal, then the scarcity of cycles in the associated Cayley graph provides a good lower bound for the size of small balls in terms of $D$ --  see Lemma~\ref{lemma:expansion}. We remark that Lemmas~\ref{lemma:isop} and \ref{lemma:expansion} are specific to Cayley graphs and are the reason why do not prove the main result for transitive graphs.

Finally, we would like to highlight some key differences with \cite{DGRSY20}.
First, we consider the (perhaps simpler) site percolation model $\{\varphi>-1\}$ instead of the Bernoulli bond percolation model on random environment used in \cite{DGRSY20}. This choice slightly simplifies the construction of the interpolation scheme and allows us to directly obtain a uniform upper bound for the site percolation threshold $p_c^{\text{site}}$ -- the classical inequality $p_c^{\text{site}}\leq p_c^{\text{site}}\leq 1-(1-p_c^{\text{bond}})^D$ \cite[Theorem 1.33]{MR1707339} gives an upper bound for $p_c^{\text{site}}$ in terms of $p_c^{\text{bond}}$ that depends on $D$. Second, we integrate out \emph{all} the fields in the decomposition of $\varphi$ through interpolation, instead of integrating out only the fields with large range and then applying the domination of finite-range models by product measure of \cite{LigSchSta97}. This is crucial as the result of \cite{LigSchSta97} inevitably depends on the degree $D$. Another important difference is that, as mentioned above, it is essential for us to prove a uniform bound on the return probabilities $p_n(x,x)$ that becomes better as $D\to\infty$. However, this is not possible for $n=0$ since $p_0(x,x)=1$ regardless of $D$. As a consequence, while integrating out the very first field (which is $1$-dependent), we need to perform a local surgery whose cost does not depend on the degree -- see Lemma~\ref{lemma:dif}. This is the main technical difference between the local surgery performed in our proof and that of \cite{DGRSY20}.

\quad

\textbf{Acknowledgments.} We would like to thank Hugo Duminil-Copin for inspiring discussions. This research was supported by the Swiss National Science Foundation and the NCCR SwissMAP.

\section{Preliminaries}\label{sec:preliminaries}

In this section, we will recall some definitions and important results. 
Given any graph $G=(V,E)$, we will denote by $B(x,r)$ the ball of radius $r\geq0$ centered at $x\in \Gamma$ for the graph distance in $G$. When necessary, we shall stress the dependency on $G$ by writing $B_G(x,r)$ instead. Given a subset $K\subset V$, we may consider its \emph{edge boundary} $\partial_E K\coloneqq \{e=\{x,y\}:~ x\in K \text{ and } y\notin K\}$ and its (inner) \emph{vertex boundary} 
$\partial_V K\coloneqq\{x\in K \mid \exists y\notin K \text{ such that } y\sim x\}$. 

Given a finitely generated group $\Gamma$ and $S$ a finite symmetric (i.e. $S^{-1}\coloneqq \{g^{-1} \mid g\in S\}=S$) generating set, one can define the Cayley graph $G=\text{Cay}(\Gamma,S)$ of $\Gamma$ with respect to $S$ as the undirected graph with vertex set $\Gamma$ and edge set consisting on all pairs $\{x, xg\}$ whenever $x\in \Gamma$ and $g\in S$.
We may also label each directed edge $(x, xg)$ by $g$. Throughout, a \emph{word} will be an expression $g_1g_2\ldots g_n$ where each $g_1,g_2,\ldots,g_n$ is an element of $S$. 
We call $S$ \emph{minimal} if it generates $\Gamma$ but no proper subset of $S$ does, and in that case the associated Cayley graph $G=\text{Cay}(\Gamma,S)$ is called a \emph{minimal} Cayley graph. 

Finally, we recall that every Cayley graph $G=\text{Cay}(\Gamma,S)$ has a well defined volume growth dimension $\mathrm{Dim}(G)\in\N\cup\{\infty\}$ (which in fact depends on $\Gamma$ only). More precisely, it follows from \mbox{Gromov's} theorem \cite{gromov1981groups} that every finitely generated group has either superpolynomial volume growth (i.e.~$\lim_{n\to} \log(|B(o,n)|)/\log{n} =\infty$), in which case we write $\mathrm{Dim}(G)=\infty$, or it has a polynomial volume growth with exponent $d\in \N$ (i.e.~there exist constant $c,C\in(0,\infty)$ such that $cn^d\leq |B(o,n)|\leq Cn^d$ for all $n\geq1$), in which case we write $\mathrm{Dim}(G)=d$.

\subsection{A finitary version of Gromov's theorem}\label{sec:finitary}

We will now state some results on the structure of Cayley graphs of polynomial growth obtained by Breuillard, Green \& Tao \cite{BGT12}. The first one is a finitary version of Gromov's theorem. 

\begin{theorem}[{\cite[Corollary 11.5]{BGT12}}]\label{theorem:nil}
For every $k>0$, there is a constant $N=N(k)$ such that the following holds. Consider a group $\Gamma$ and let $S$ be a generating set of $\Gamma$. Suppose that $|B(o,n)|\leq |S|n^k$ in the Cayley graph $\text{Cay}(\Gamma,S)$ for some $n\geq N$. Then there is a normal subgroup $H \lhd \Gamma$ such that $\Gamma/H$ contains a nilpotent subgroup of index at most $N$.
\end{theorem}

The following result states that polynomial growth at a single large enough radius implies polynomial growth at every larger radius. This will allow us to extract uniform estimates on the growth rate of Cayley graphs of large enough dimension.

\begin{theorem}[{\cite[Corollary 11.9]{BGT12}}]\label{theorem:growth}
For every $k>0$, there are constants $N=N(k)$ and $d=d(k)$ such that the following holds. Consider a group $\Gamma$ and let $S$ be a generating set of $\Gamma$. Suppose that $|B(o,n)|\leq |S|n^k$ in the Cayley graph $\text{Cay}(\Gamma,S)$ for some $n\geq N$. Then $|B(o,r)|\leq |S|r^d$ for every $r\geq n$.
\end{theorem}

\subsection{Inequalities for $p_c$}

We will now recall some inequalities for $p_c$, which we will later on combine with \Tr{theorem:nil} to obtain a uniform bound on $p_c$ for low-dimensional Cayley graphs.
The following result of Hutchcroft \& Tointon \cite{HT21} gives a uniform non-trivial upper bound on $p_c$ for Cayley graphs containing a nilpotent subgroup of bounded index.

\begin{theorem}[{\cite[Theorem 3.20]{HT21}}]\label{thm:inf.perc.nilp}
For each $n\geq 1$ there exists $\varepsilon=\varepsilon(n)>0$ such that if $\Gamma$ is a finitely generated group of superlinear growth that contains a nilpotent subgroup of index at most $n$, and if $S$ is a finite generating set of $\Gamma$, then $p_c(\text{Cay}(\Gamma,S))\le1-\varepsilon$.
\end{theorem}

Let $\Gamma$ be a group of automorphisms of a graph $G=(V,E)$. The quotient graph $G/\Gamma$ is the graph whose vertices are the orbits $V/\Gamma\coloneqq \{\Gamma v :v \in V\}$, and an edge $\{\Gamma u, \Gamma v\}$ is contained in $G/\Gamma$ if there are representatives $u_0 \in \Gamma u, v_0 \in \Gamma v$ such that $\{u_0, v_0\} \in E$. We state below a classical result of Benjamini \& Schramm \cite{BenSch96} showing that $p_c$ is monotonic under quotients. We remark that strict monotonicity can also be proved under certain assumptions \cite{MS19}. 

\begin{theorem}[{\cite[Theorem 1]{BenSch96}}]\label{theorem:quotient}
Let $G_2$ be a quotient of a graph $G_1$. Then 
$$p_c(G_1)\leq p_c(G_2).$$
\end{theorem}

\subsection{Isoperimetry and heat kernel bounds}

Given a regular (i.e.~with constant degree) graph $G$, its \emph{expansion profile} is the function $\Phi:(1,\infty)\to \R$ defined as
$$\Phi(u)\coloneqq \min\left\{\frac{|\partial_E K|}{D|K|} \mid 0<|K|\leq u\right\},$$
where $D$ denotes the degree of $G$. Notice that for a Cayley graph $G=\text{Cay}(\Gamma,S)$, we simply have $D=|S|$. The following result \cite{LyoPer17} translates any lower bound on $\Phi$ to an upper on the heat kernel $p_n(x,y)$ for the simple random walk on $G$. Suitable extensions that apply to any reversible Markov chain are possible, but for simplicity we state the result only in the special case of simple random walk on a regular graph.

\begin{theorem}[{\cite[Theorem 6.31]{LyoPer17}}]\label{theorem:pn}
Let $G$ be a regular graph. For every $\varepsilon>0$, if 
$$n\geq 1+\int_1^{4/\varepsilon} \dfrac{16du}{u \Phi^2(u)},$$
then $p_n(x,y)\leq \varepsilon$.
\end{theorem}

In the special case of Cayley graphs, it is well-known that lower bounds on the growth rate implies an isoperimetric inequality involving the \emph{vertex} boundary.

\begin{theorem}[{\cite[Lemma 7.2]{LyoMorSch08}}]\label{theorem:isop}
Let $G$ be a Cayley graph and define $R(m)\coloneqq\min\{n\geq 1 \mid |B(o,r)|\geq m\}$. Then for every non-empty finite set $K\subset V$, we have
$$\dfrac{|\partial_V K|}{|K|}\geq \dfrac{1}{2R(2|K|)}.$$
\end{theorem}

An extension of the above isoperimetric inequality to transitive graphs was obtained in \cite{TTfinitary}. In the next section we will prove a version of theorem Theorem~\ref{theorem:isop} -- see Lemma~\ref{lemma:isop} -- that involves the \emph{edge} boundary instead, which is more suitable for applying Theorem~\ref{theorem:pn}.

\section{Proof of main result}\label{sec:proof}

In this section, we will prove \Tr{thm:main}. 
As in \cite{DGRSY20}, we will use the Gaussian free field (GFF) to handle the high-dimensional Cayley graphs.
Given a transient graph $G=(V,E)$, the Gaussian free field $\varphi = (\varphi_x)_{x\in V}$  on $G$ is defined as the centered Gaussian process with covariance
$\Ex(\varphi_x \varphi_y) = g(x,y)$ for all $x,y\in V$,
where $g(\cdot,\cdot)$ stands for the Green function of the simple random walk on $G$.
Given $h\in \R$, we consider the excursion set
$$\{\varphi> h\}\coloneqq  \{x\in V:\, \varphi_x>h\}$$ 
seen as a random subgraph of $G$ (with the induced adjacency).
As $h$ varies, this defines a percolation model for which one may expect to see a phase transition in $h$ from a percolative regime -- where $\{\varphi> h\}$ contains an infinite connected component -- to a non-percolative regime -- where all the clusters of $\{\varphi> h\}$ are finite.
We can then define the percolation \emph{critical point} 
$h_*\coloneqq \sup\{h\in \R:\,\Pr[o\xleftrightarrow[]{\varphi> h}\infty]> 0\},$
where $\{o\xleftrightarrow[]{\varphi> h}\infty\}$ denotes the event that
a fixed origin $o\in V$ belongs to an infinite cluster in $\{\varphi> h\}$. More generally, $A\xleftrightarrow[]{\varphi> h} B$ denotes the event that there is a path in $\{\varphi> h\}$ connecting  a vertex in $A$ to a vertex in $B$. A soft argument due to Bricmont, Lebowitz \& Maes \cite{BricmontLJM87} shows that the GFF percolates above any negative level, i.e.\ 
$h_*\geq 0$ for every transient graph (in particular for any Cayley graph $G$ with $\mathrm{Dim}(G)>2$). In particular, we always have $$\Pr[o\xleftrightarrow[]{\varphi> -1}\infty]>0.$$

\Tr{thm:main} for $G$ being a high-dimensional Cayley graph will follow from the next result. In what follows, we use the notation $A\longleftrightarrow B$ to denote the event that there is an open path in $\omega$ connecting a vertex in $A$ to a vertex in $B$.

\begin{theorem}\label{theorem:comparison}
There are constants $\varepsilon>0$ and $d_0>2$ such that for every minimal Cayley graph $G=\text{Cay}(\Gamma, S)$ with $\mathrm{Dim}(G)\geq d_0$ the following holds. For every $A\subset \Lambda\subset V$, one has
$$\Pf_{1-\varepsilon}[A\longleftrightarrow \Lambda^c]\geq \Pr[A\xleftrightarrow[]{\varphi> -1}\Lambda^c].$$
\end{theorem}

Before proving \Tr{theorem:comparison}, we shall deduce \Tr{thm:main} from it.

\begin{proof}[Proof of \Tr{thm:main}]
	We start by treating the high-dimensional case. Let $\varepsilon>0$ and $d_0>2$ be the constants of \Tr{theorem:comparison}. Fix a Cayley graph $G=\text{Cay}(\Gamma,S)$ with $\mathrm{Dim}(G)\geq d_0$, and let $S'\subset S$ be a minimal (symmetric) generating set of $\Gamma$. Applying Theorem~\ref{theorem:comparison} for the minimal Cayley graph $G'=\text{Cay}(\Gamma,S')$ and for $A=\{o\}$ and $\Lambda=B_{G'}(o,n)$, and then letting $n\to\infty$, we deduce that 
	$\Pf_{1-\varepsilon}[o\longleftrightarrow \infty] \geq \Pr[o\xleftrightarrow[]{\varphi> -1}\infty]>0$ on $G'$. Since $G'$ is a subgraph of $G$, we conclude that $p_c(G)\leq p_c(G')\leq 1-\varepsilon$.
	
	Consider now a Cayley graph $G=\text{Cay}(\Gamma,S)$ of superlinear growth with $\mathrm{Dim}(G)< d_0$. Using \Tr{theorem:nil} we obtain a normal subgroup $H \lhd \Gamma$ such that $\Gamma/H$ contains a nilpotent subgroup of index at most $N$ for some constant $N=N(d_0)$. Since the quotient graph $G/H$ is simply the Cayley graph $\text{Cay}(\Gamma/H,S/H)$, it follows from \Tr{thm:inf.perc.nilp} that $p_c(G/H)\leq 1-\varepsilon'$ for some constant $\varepsilon'=\varepsilon'(d_0)>0$. We can now apply \Tr{theorem:quotient} to conclude that $p_c(G)\leq p_c(G/H)\leq 1-\varepsilon'$. This completes the proof.
\end{proof}

Our first step towards proving \Tr{theorem:comparison} is to introduce a family of percolation models which interpolate between $\{\varphi> -1\}$ and Bernoulli site percolation of a certain parameter close enough to $1$, in such a way that the connectivity probabilities do not decrease along the way. To define these models, we first need to decompose $\varphi$ into finite-range independent Gaussian fields. 
For any $x,y\in V$, set 
$$g_n(x,y)\coloneqq \sum_{2^n-2\leq k<2^{n+1}-2} p_k(x,y)$$
for all $n\geq 1$.
The matrices $(g_n)_n$ satisfy the following properties:
\begin{enumerate}
\item \label{decomp.G}
$g(x,y)=\sum_{n\geq 1} g_n(x,y)$ for all $x,y\in  V$,
\item \label{cov.G}
$g_n$ is a covariance matrix (i.e.~symmetric positive semi-definite) for every $n\geq 1$,
\item \label{posit.assoc}
$g_n(x,y)\geq 0$ for any $x,y\in V$ and $n\geq 1$,
\item \label{range.G}
$g_n(x,y)=0$ for any $x,y\in V$ with $d(x,y)> 2^{n+1}-2$.
\end{enumerate}
Properties \ref{decomp.G}, \ref{posit.assoc} and \ref{range.G} follow directly from the definitions. To verify Property \ref{cov.G}, we only need to show that $g_n$ is positive semi-definite. Indeed, let us first assume that $G$ is a finite graph. Notice that 
$g_n=\sum_{2^n-2\leq k<2^{n+1}-2} p_1^k,$ 
hence the eigenvalues of $g_n$ are of the form $\sum_{2^n-2\leq k<2^{n+1}-2} \lambda^k$, where $\lambda$ is some eigenvalue of $p_1$. Since $p_1$ is a stochastic matrix, we have that $-1\leq\lambda\leq 1$ by the Perron--Frobenius theorem. This implies that 
$\sum_{2^n-2\leq k<2^{n+1}-2} \lambda^k=(1+\lambda)\sum_{2^{n-1}-1\leq k<2^{n}-1} \lambda^{2k}\geq 0,$ 
hence $g_n$ is positive semi-definite. An approximation argument can be used to deduce it for infinite graphs as well.
It follows from Properties \ref{decomp.G} and \ref{cov.G} above that, if $\phi^n\sim \mathcal{N}(0,g_n)$ are \emph{independent} Gaussian fields and $G$ is transient, then 
\begin{equation}
\varphi=\sum_{n\geq 1}\phi^n
\end{equation}
in law (convergence of the series in $L^2$ and almost surely can be proved by the martingale convergence theorem, for example).

Our models will be defined as the superposition of a Bernoulli site percolation with appropriate excursion sets of the fields $\phi^n$. Given some $t\geq0$, let $\omega^0$ be a Bernoulli site percolation configuration of parameter $1-e^{-t}$. Fix $\lambda_1\coloneqq -1-\pi^2/6$ and $\lambda_n \coloneqq 1/(n-1)^2$ for every $n\geq 2$. Given $t\geq0$, an integer $n\geq 1$ and a parameter $\lambda \geq \lambda_n$, we let $\bm{\omega}$ be the superposition of $\omega^0$ and $\{\phi^n> \lambda\}\cup \bigcup_{k\geq n+1} \{\phi^k > \lambda_k\}$.
We write $\Pr_{t,n,\lambda}$ for the law of $\bm{\omega}$.

Notice that $\{\varphi> -1\}\subset \bigcup_{n\geq 1} \{\phi^n > \lambda_n\}$, which implies
$$\Pr_{0,1,\lambda_1}[A\longleftrightarrow\Lambda^c]\geq\Pr[A\xleftrightarrow[]{\phi> -1}\Lambda^c].$$

We will now proceed to show that for a certain function $t=t(n,\lambda)>0$, the connection probabilities $\Pr_{t(n,\lambda),n,\lambda}[A\longleftrightarrow\Lambda^c]$ is non-decreasing in $\lambda$ and $n$, while $t(n,\lambda)$ remains bounded.  
For that, we will compare, for each $n\geq1$, the partial derivative with respect to $t$ and $\lambda$.

Let us first recall Russo's formula for derivatives of events in Bernoulli percolation and the notion of pivotality. Consider an increasing event $E$ depending on finitely many edges. A set $K$ of vertices in $V$ is {\em pivotal} (in $\bm{\omega}$) for $E$ if the configuration is in $E$ when all vertices in $K$ are open and is not in $E$ when all these vertices are closed. We say that $E$ is {\em open} (resp.~{\em closed}) {\em pivotal} if in addition $\bm{\omega}\in E$ (resp.~$\bm{\omega}\notin E$). Russo's formula states that for every event $E$ depending on finitely many vertices,
	\begin{equation}\label{eq:d1}
	\frac{\rm d}{{\rm d}t}\Pr_{t,n,\lambda}[E]=\sum_{x\in V}\Pr_{t,n,\lambda}[x\text{ closed pivotal for }E].
	\end{equation}
For the derivative in $\lambda$, a quick analysis of $\frac{1}{\delta}\left(\Pr_{t,n,\lambda+\delta}[E]-\Pr_{t,n,\lambda}[E]\right)$ gives that 
	\begin{equation}\label{eq:d2}
	-\frac{\rm d}{{\rm d}\lambda}\Pr_{t,n,\lambda}[E]=\rho^n(\lambda) \sum_{x\in V} \Pr_{t,n,\lambda}[x\text{ closed pivotal for }E | \phi_x^n=\lambda],
	\end{equation}
where $\rho^n(\cdot)$ is the density of any $\phi^n_x$ (which does not depend on $x$ by transitivity).

\begin{lemma}\label{lemma:dif}
There exists a universal constant $C_0>0$ such that the following holds for every graph $G$ of maximal degree $D$. For every $n\geq1$, $\lambda\geq\lambda_n$, every $t\geq \log(2)$ and every finite subsets $A\subset\Lambda$ of $V$, we have
$$-\dfrac{\rm d}{{\rm d}\lambda}\Pr_{{t},n,\lambda}[A\longleftrightarrow \Lambda^c]\leq C_0\rho^n(\lambda) \left(16D\right)^{L_n-1}\dfrac{\rm d}{{\rm d}t}\Pr_{{t},n,\lambda}[A\longleftrightarrow \Lambda^c],$$  
where $L_n\coloneqq2^{n+1}-3$. In particular, for $n=1$ we have 
$$-\dfrac{\rm d}{{\rm d}\lambda}\Pr_{{t},1,\lambda}[A\longleftrightarrow \Lambda^c]\leq C_0\rho^1(\lambda)\dfrac{\rm d}{{\rm d}t}\Pr_{{t},1,\lambda}[A\longleftrightarrow \Lambda^c].$$
\end{lemma}
\begin{proof}
To prove the differential inequality, we will apply \eqref{eq:d1} and \eqref{eq:d2} for $E=\{A\longleftrightarrow \Lambda^c\}$. We will write ``pivotal'' instead of ``pivotal for $E$''.
The proof is similar to that of \cite{DGRSY20}.

Consider a vertex $x$ and let $Q_A\coloneqq \partial_V B(x,L_n) \cup (A \cap B(x,L_n))$ and $Q_{\Lambda}\coloneqq \partial_V B(x,L_n) \cup (\Lambda^c \cap B(x,L_n))$.
Observe that when $x$ is closed pivotal, there are vertices $z\in Q_A$ and $w\in Q_{\Lambda}$ that are connected to $A$ and $\Lambda^c$, respectively, via an open path that visits only one vertex of $B(x,L_n)$ (namely $z$ and $w$, respectively). 
Now fix an arbitrary order on the vertex set $V$ and denote by $\mathcal{E}(z,w)$ the event that $z$ and $w$ are the smallest vertices with the above property. Therefore the events $\mathcal{E}(z,w)$ are disjoint and 
$$\{x\text{ closed pivotal}\}\subset \bigcup_{z,w\in B(x,L_n)} \mathcal{E}(z,w)\cap\mathcal{C},$$ 
where $\mathcal{C}=\{A\centernot\longleftrightarrow \Lambda^c\}\coloneqq\{A\longleftrightarrow \Lambda^c\}^c$. Notice that $\mathcal{E}(z,w)$ brings some positive information about the values of $\phi^n_y$, $y\in B(x,L_n)$ which we would like to remove. To this end, let $\mathcal{O}(z,w)$ be the intersection of $\mathcal{E}(z,w)$ with the event $\{\text{$z$ and $w$ are open in $\omega^0$}\}$. One can easily compare the probabilities of the two events as follows. Given a configuration $\bm{\omega}$ such that $\mathcal{E}(z,w)$ happens, we define $\bm{\omega'}$ from $\bm{\omega}$ by making $z$ and $w$ open in $\omega^0$, which is at most four-to-one map. Therefore
$$\Pr_{{t},n,\lambda}[\mathcal{E}(z,w)\cap\mathcal{C} | \phi^n_x=\lambda]\leq \dfrac{4}{p^2} \Pr_{{t},n,\lambda}[\mathcal{O}(z,w)\cap\mathcal{C}| \phi^n_x=\lambda],$$
where $p=1-e^{-t}$ is the density of $\omega^0$.

Now for every $\mu<\lambda$, $\phi^n$ conditioned on $\phi^n_x=\lambda$ and $\phi^n$ conditioned on $\phi^n_x=\mu$ are shifts of the same centered Gaussian process, and the difference between the two shift functions is equal to $(\lambda-\mu)g_n(x,y)/g_n(x,x)$. The latter is non-negative for $y\in B(x,L_n)$ and equal to $0$ for $y\notin B(x,L_n)$ (by Properties \ref{posit.assoc} and \ref{range.G} of $(g_n)$, respectively), hence we obtain 
$$\Pr_{{t},n,\lambda}[\mathcal{O}(z,w)\cap\mathcal{C} | \phi_x^n=\lambda] \leq \Pr_{{t},n,\lambda}[\mathcal{O}(z,w)\cap \mathcal{C}| \phi_x^n=\mu]$$
for every $\mu<\lambda$. 
Integrating on $\mu< \lambda$ gives
$$\Pr_{{t},n,\lambda}[\mathcal{O}(z,w)\cap\mathcal{C} | \phi_x^n=\lambda] \leq \dfrac{\Pr_{{t},n,\lambda}[\mathcal{O}(z,w)\cap\mathcal{C}]}{\Pr[\phi^n_x< \lambda]}.$$
To estimate the denominator, notice that $\Pr[\phi^n_x<\lambda]\geq \Pr[\phi^n_x\leq \lambda_n]$. The latter is at least $1/2$ for $n\geq 2$ and equal to a constant $0<a\leq 1/2$ for $n=1$, as $\phi^1_x$ is a standard normal random variable. Thus 
$$\Pr_{{t},n,\lambda}[\mathcal{O}(z,w)\cap\mathcal{C} | \phi_x^n=\lambda] \leq \dfrac{1}{a}\Pr_{{t},n,\lambda}[\mathcal{O}(z,w)\cap\mathcal{C}].$$

On the event $\mathcal{O}(z,w)\cap \mathcal{C}$ we can create a pivotal vertex as follows. We fix a path $\gamma$ of length at most $2L-2$ in $B(x,L_n-1)$ starting from a neighbour of $z$ and ending at a neighbour of $w$, and define a configuration $\bm{\omega'}$ by opening in $\omega^0$ the vertices of $\gamma$ one by one until the first time that a vertex $u$ of $B(x,L_n-1)$ becomes pivotal. The cost of opening all these vertices is at most $(1/p)^{2L_n-2}$, and the corresponding map is at most $2^{2L_n-2}$-to-one. Hence
\begin{align*}
\Pr_{{t},n,\lambda}[x&\text{ closed pivotal} | \phi^n_x=\lambda]\\
&\leq \dfrac{4}{ap^2}\sum_{z,w\in B(x,L)} \sum_{u\in B(x,L_n-1)} \left(\dfrac{2}{p}\right)^{2L_n-2} \Pr_{{t},n,\lambda}[\mathcal{O}(z,w), u \text{ closed pivotal}] \\
&\leq \dfrac{1}{a}\left(\dfrac{2}{p}\right)^{2L}\sum_{u\in B(x,L_n-1)} \Pr_{{t},n,\lambda}[u \text{ closed pivotal}],
\end{align*}
where in the last line we used the fact that the events $\mathcal{O}(z,w)$ are disjoint.

Recalling the differential formulas \eqref{eq:d1} and \eqref{eq:d2} and noticing that the number of vertices in $G$ that are at distance at most $L_n-1$ from a vertex $u$ is at most $D^{L_n-1}$, we obtain
$$-\dfrac{\rm d}{{\rm d}\lambda}\Pr_{{t},n,\lambda}[A\longleftrightarrow \Lambda^c]\leq \dfrac{\rho^n(\lambda)}{a}\left(\dfrac{2}{p}\right)^{2L_n}D^{L_n-1}\dfrac{\rm d}{{\rm d}t}\Pr_{{t},n,\lambda}[A\longleftrightarrow \Lambda^c].$$		
The result follows with $C_0=16/a$ by reminding that $p=1-e^{-t}\geq 1/2$.
\end{proof}

Our next step is to estimate $\rho^n(\lambda)$ and for that, we will need to estimate the return probabilities $p_n(x,x)$. To this end, consider some $k>0$ and recall Theorem \ref{theorem:growth}. Notice that if $G=\text{Cay}(\Gamma,S)$ is a Cayley graph with $\mathrm{Dim}(G)> d$, then $|B(o,n)|\geq Dn^k$ for every $n\geq N$, where $N=N(k)$ and $d=d(k)$ are the constants of the theorem. Hence $|B(o,n)|\geq cn^k$ for every $n\geq 1$, where $c=N^{-k}$. This shows that 
$$R(m)\leq \left(\dfrac{m}{c}\right)^{\frac{1}{k}} \text{ for every } m\geq1,$$
and applying \Tr{theorem:isop} we obtain that
$$\Phi(u)\geq \dfrac{c^{\frac{1}{k}}}{2D(2u)^{\frac{1}{k}}} \text{ for every } u>1.$$
Plugging this inequality to \Tr{theorem:pn} for $\varepsilon=D^k m^{-k/2}$, we obtain that $p_n(x,x)\leq \varepsilon$ for every $n\geq Cm$ and some constant $C=C(k)$, hence
\begin{equation}\label{eq:gen_bound}
p_n(x,x)\leq \dfrac{C'D^k}{n^{\frac{k}{2}}} \text{ for every } n\geq1.
\end{equation} for another constant $C'=C'(k)$.
This upper bound has the disadvantage of getting worse as $D\to\infty$, which is an important obstacle towards proving a uniform upper bound for $p_c$. We remark that once can easily obtain an upper bound on $p_n(x,x)$ that gets better as $D\to\infty$ for every large enough $n$ (depending on $D$) at the expense of decreasing the exponent of $n$. Indeed, taking $k=\frac{2r^2+2}{r-2}$ we obtain that for some constant $C''=C''(k)$ one has
\begin{equation}\label{eq:d3}
p_n(x,y)\leq \dfrac{C''}{Dn^r} \text{ for every } n\geq D^r.
\end{equation} 
In the next theorem, we show that the above inequality holds for every $n\geq 1$.

\begin{theorem}\label{theorem:uni_pn}
Let $r>0$. There are constants $C_1=C_1(r)>0$ and $d=d(r)>0$ such that if $G=\text{Cay}(\Gamma,S)$ is a minimal Cayley graph with $\mathrm{Dim}(G)\geq d$, then
$$p_n(x,y)\leq \dfrac{C_1}{Dn^r} \text{ for every } n\geq 1 \text{ and every } x,y\in V.$$
\end{theorem}

Before proving the theorem above, we will use it to prove \Tr{theorem:comparison}, which we will use in turn to prove \Tr{thm:main}.

\begin{proof}[Proof of \Tr{theorem:comparison}]
Let $d_0=d(3)$ be the constant in the statement of \Tr{theorem:uni_pn} for $r=3$. We first claim that there is a universal constant $M>0$ such that for every minimal Cayley graph $G=\text{Cay}(\Gamma,S)$ with $\mathrm{Dim}(G)\geq d_0$, one has 
\begin{equation}\label{eq:M}
t_\infty=t_\infty(G)\coloneqq\log(2) + C_0\sum_{n=1}^\infty (16D)^{L_n-1} \int_{\lambda_n}^\infty \rho^n(\lambda) d\lambda\leq M,
\end{equation}
where $L_n=2^{n+1}-3$. Indeed, recall that $\rho^n(\lambda)$ is simply the density of $\phi^n_x$, which in turn is a centered Gaussian with variance $g_n(x,x)$, therefore
$$\int_{\lambda_n}^\infty \rho^n(\lambda) d\lambda = \Pr[N\geq \lambda_n/\sqrt{g_n(x,x)}],$$
where $N$ is a standard normal random variable. 
Now, for $n=1$ we have $g_1(x,x)=1$ and for $n\geq 2$ we have 
$$g_n(x,x)\leq \dfrac{C'_1}{DL_n^2}$$ 
by \Tr{theorem:uni_pn}, where $C'_1>0$ is a constant. Finally, recalling the definition of $\lambda_n$, we see that $(16D)^{L_n-1}\int_{\lambda_n}^\infty \rho^n(\lambda) d\lambda$ is a constant independent of $D$ for $n=1$ and is bounded from above by $C_2\exp\{-c_2DL_n^{2}/n^4\}$ for $n\geq 2$. The claim follows readily.

We will now use \eqref{eq:M} to obtain a uniform non-trivial upper bound for $p_c$. For every $n\geq1$ and $\lambda\geq \lambda_n$, consider 
$$t(n,\lambda)\coloneqq \log(2) + C_0(16D)^{L_n-1} \int_{\lambda_n}^\lambda \rho^n(\alpha) d\alpha + C_0\sum_{k=1}^{n-1} (16D)^{L_k-1} \int_{\lambda_k}^\infty \rho^n(\alpha) d\alpha.$$
Then one directly deduces from Lemma~\ref{lemma:dif} that for every finite subsets $A\subset \Lambda$ of $V$, the function
$$(n,\lambda)\mapsto \Pr_{{t(n,\lambda)},n,\lambda}[A\longleftrightarrow \Lambda^c]$$
is increasing with respect to the natural order given by $(n,\alpha)\leq (m,\beta)$ if and only if $n<m$ or $n=m$ and $\alpha\leq\beta$. Now simply notice that $\Pr_{t(1,\lambda_1),1,\lambda_1}$ stochastically dominates $\{\varphi>-1\}$, while $\Pr_{{t(n,\lambda)},n,\lambda}\to \Pf_p$ as $(n,\lambda)\to\infty$, where $p=1-e^{-t_\infty}\leq p_0\coloneqq 1-e^{-M}<1$. Therefore
$$\Pf_{p_0}[A\longleftrightarrow\Lambda^c]\geq\Pr[A\xleftrightarrow[]{\varphi> -1}\Lambda^c].$$
\end{proof}

It now remains to prove \Tr{theorem:uni_pn}. To this end, we will prove a new isoperimetric inequality and then proceed to obtain a local expansion result for minimal Cayley graphs. 

Let us define $\mathcal{B}(n)=\min_H \{|B_H(o,n)|\}$, where the minimum ranges over all Cayley graphs $H=\text{Cay}(\Gamma',S')$ with $\Gamma'<\Gamma$ generated by $S'\subseteq S$ with $|S'|\geq|S|/2=D/2$. Notice that $H$ is in particular a subgraph of $G$. We also define $\overline{R}(m)=\min\{n\geq 1 \mid \mathcal{B}(n)\geq m\}$, with the understanding that $\overline{R}(m)$ is infinite if $\mathcal{B}(n)<m$ for every $n\geq 1$. 
In the following lemma, we prove an isoperimetric inequality for the \emph{edge} expansion of finite sets, which is an improvement on Theorem~\ref{theorem:isop}. We use the convention $1/\infty=0$. 

\begin{lemma}\label{lemma:isop}
Let $G=\text{Cay}(\Gamma,S)$ be a Cayley graph. Then for every finite set $\emptyset \neq K\subset V$,
$$\frac{|\partial_E K|}{D|K|}\geq \dfrac{1}{16\overline{R}(2|K|)}$$
\end{lemma}
\begin{proof}
Let $r=\overline{R}(2|K|)$. If $r=\infty$, then there is nothing to prove, so let us assume that $r$ is finite.

We start by the following definition and a useful observation. For every $g\in V$, let $N(g)$ be the number of vertices $x\in K$ such that $xg \in V\setminus K$. If $g=g_1\ldots g_r$, $g_i\in S$ is a word of length $r$, then we have that for some $i=1,\ldots,r$, the number of vertices $x\in K$ such that $xg_1\ldots g_{i-1}\in K$ but $xg_1\ldots g_i\in V\setminus K$ is at least $N(g)/r$. Since the map $x\rightarrow xg_1\ldots g_{i-1}$ is injective, we deduce that $N(g_i)\geq N(g)/r$.

The next step is to find vertices $g\in V$ for which $N(g)$ is large. Consider some $x\in K$ and choose some $g\in B(o,r)$ uniformly at random. The probability that $xg\in V \setminus K$ is at least $1/2$ by our choice of $r$, hence the expected number of elements of $K$ that are mapped outside of $K$ is at least $|K|/2$. This implies that there is some $g\in B(o,r)$ such that $N(g)\geq |K|/2$. From the above observation, we obtain that there is some $g_1\in S$ such that $N(g_1)\geq \frac{|K|}{2r}$. In particular, the number of directed edges between $K$ and $V\setminus K$ labelled by $g_1$ is at least $\frac{|K|}{2r}$. Consider now the subgroup $H$ generated by $S\setminus \{g_1,g_1^{-1}\}$. We argue as above with $B(o,r)$ replaced by $B_H(o,r)$ to obtain some $g_2\in S$, $g_2\neq g_1$ such that the number of directed edges between $K$ and $V\setminus K$ labelled by $g_2$ is at least $\frac{|K|}{2r}$. Proceeding inductively, we find $m=\max\{\left\lfloor \frac{D}{4} \right \rfloor, 1\}$ distinct elements $g_1,\ldots,g_m$ such that for every $i=1,\ldots,m$, the number of directed edges between $K$ and $V\setminus K$ labelled by $g_i$ is at least $\frac{|K|}{2r}$. Therefore,
$$\dfrac{|\partial_E K|}{D|K|}\geq \dfrac{\max\{\left\lfloor \frac{D}{4} \right \rfloor, 1\}}{2Dr}\geq \max\left\{\dfrac{1}{8r}-\dfrac{1}{2Dr},\dfrac{1}{2Dr}\right\}\geq \dfrac{1}{16r},$$
since the corresponding undirected edges are all distinct.
\end{proof}

In the following lemma, we obtain a polynomial lower bound for $|B(o,n)|$ in terms of $D$. The reader can think of $n$ as being fixed. 

\begin{lemma}\label{lemma:expansion}
There is a sequence $(c_n)_{n\geq1}$ of positive real numbers such that for every minimal Cayley graph $G=\text{Cay}(\Gamma,S)$ and every $n\geq 1$, one has $|B(o,n)|\geq c_nD^n$.
\end{lemma}
\begin{proof}
Among the vertices of $G$ at distance $n$ from the identity, let $S_n$ be the set of words $g_1g_2\ldots g_n\in \partial_V B(o,n)$ with $g_i\neq g_j,g_j^{-1}$ for every $i\neq j$. We will prove inductively that there is a sequence $(t_n)$ such that $|S_n|\geq t_nD^n$ whenever $D\geq 4n$. To handle both cases $D\geq 4n$ and $D<4n$, we will then let $c_n\coloneqq\min\{t_n,(4n)^{-n}\}$. 

For $n=1$ we can set $t_1=1$ and there is nothing to prove, so let us assume that the inductive statement is true for some $n\geq 1$ and we will prove it for $n+1$. To this end, let us first consider a vertex $v=g_1\ldots g_n\in S_n$. We will count the number of elements $g_{n+1}\in S$ such that $vg_{n+1}$ coincides with some vertex $u=h_{1}\ldots h_{m}\in B(o,n)$ (notice that we can assume $m\leq n$). 

We claim that either $g_{n+1}=g_i$ or $g_{n+1}=g_i^{-1}$ for some $i=1,\ldots,n$. Indeed, assume that this is not the case. Then for $1\leq i,j\leq n+1$, $i\neq j$ we have $g_i\neq g_j,g_j^{-1}$. Since there are $n+1$ generators in $\{g_1,\ldots,g_{n+1}\}$ and at most $n$ generators in $\{h_1,\ldots, h_m\}$, we can deduce that there is $i\in \{1,\ldots,n+1\}$ such that $g_i\neq h_j,h_j^{-1}$ for every $j\in \{1,\ldots,m\}$. Solving the relation $g_1\ldots g_{n+1}=h_{1}\ldots h_{m}$ for $g_i$, we can deduce that $g_i$ is expressed in terms of generators in $S\setminus\{g_i,g_i^{-1}\}$, which is a contradiction with the minimality of $S$. This proves the claim.

It follows that there are at most $2n$ possibilities for $g_{n+1}$, hence each element of $S_n$ is incident to at least $D-2n$ elements of $S_{n+1}$. In other words, there are at least $|S_n|(D-2n)$ edges between $S_n$ and $S_{n+1}$.

A priori it is possible that some vertices of $S_{n+1}$ are incident with a lot of vertices of $S_n$. We will show that this is not the case. Indeed, let us consider a word $v=g_1\ldots g_{n+1}\in S_{n+1}$ such that $ug_{n+2}$ coincides with $h_{1}\ldots h_{m}\in \partial_V B(o,n)$ for some $g_{n+2}\in S$. Arguing as above, we see that there are at most $2n+2$ possibilities for $g_{n+2}$, hence 
$$|S_{n+1}|\geq \dfrac{|S_n|(D-2n)}{2n+2}\geq \dfrac{t_n D^{n+1}}{4n+4}.$$
In the last inequality we used our inductive assumption and the fact that $D-2n\geq D/2$ for $D\geq 4n$.
We have thus proved the inductive step for $t_{n+1}\coloneqq\frac{t_n}{4n+4}$. This completes the proof.
\end{proof}

We are now ready to prove \Tr{theorem:uni_pn}.

\begin{proof}[Proof of \Tr{theorem:uni_pn}]
We assume without loss of generality that $r>2$.
We will take cases on the degree of $G$. Let $D_0\coloneqq \frac{2^{r^2+5}}{c_{r^2+2}}$, where $c_{r^2+2}$ is the constant of \Lr{lemma:expansion}. Applying \eqref{eq:gen_bound} for $k=2r$ we obtain that  
$$p_n(x,y)\leq \dfrac{C}{Dn^r} \text{ for every } n\geq 1$$ 
for all Cayley graphs $G$ of degree $D<D_0$ and sufficiently high dimension (depending on $d$).

To handle Cayley graphs of degree $D\geq D_0$, we will take cases on the value of $n$. First recall from \eqref{eq:d3} that for some constant $C=C(r)$, we have
$$p_n(x,y)\leq \dfrac{C}{Dn^r} \text{ for every } n\geq D^r$$
whenever $G$ has sufficiently high dimension (depending on $r$).
It remains to handle the case $n<D^r$. Let us start by applying \Lr{lemma:expansion} for $B_H(o,n)$ in place of $B(o,n)$, where $H$ is any subgraph of $G$ as in the definition of $\mathcal{B}(n)$, to deduce that $\mathcal{B}(r^2+2)\geq c_{r^2+2}(D/2)^{r^2+2}$. By our choice of $D_0$, $\mathcal{B}(r^2+2)\geq 8D^{r^2+1}$ whenever $D\geq D_0$, and hence $\overline{R}(m)\leq r^2+1$ for all $m\leq 8D^{r^2+1}$. We can now deduce from \Lr{lemma:isop} that 
$\frac{|\partial_E K|}{D|K|}\geq \frac{1}{16r^2+32}$ for every $K\subset V$ with $|K|\leq 4D^{r^2+1}$, or equivalently $\Phi(u)\geq \frac{1}{16r^2+32}$ for all $u\leq 4D^{r^2+1}$. Applying \Tr{theorem:pn} for $\varepsilon=1/D^{r^2+1}$ we deduce that $p_n(x,y)\leq 1/D^{r^2+1}$ for every $n\geq t\log(D)$, where $t=t(r)>0$ is a constant. In particular, $$p_n(x,y)\leq \dfrac{1}{Dn^r} \text{ for every } t\log(D)\leq n <D^r.$$ 

Finally, to handle the case $1\leq n< t\log(D)$, it is clearly enough to prove that $p_n(x,y)\leq 1/D$ for every $n\geq 1$ and $p_n(x,y)\leq 6/D^2$ for every $n\geq 4$. For the first inequality, simply notice that since $p_1(x,z)= \tfrac{1}{D} 1_{z\in xS}$, one has $p_n(x,y)=\sum_{z\in V}p_{n-1}(x,z)p_1(z,y)= \tfrac{1}{D}\sum_{z\in xS}p_{n-1}(x,z)\leq1/D$. As for the second inequality, first recall from the proof of \Lr{lemma:expansion} that every $g\in S$ has at least $D-2$ neighbours in $S_2$ and each $g'\in S_2$ has at most $4$ neighbours in $S$. Therefore, when we start a simple random walk $(X_n)_{n \geq 0}$ at $o$, either $X_2\not \in S_2$, which happens with probability at most $2/D$, or $X_2\in S_2$, and from there it has probability at most $4/D$ to reach $S$ at time $3$. This implies that $\sum_{z\in S}p_3(o,z)\leq 2/D+4/D=6/D$ and hence $p_4(o,o)=\sum_{z\in V}p_3(o,z)p_1(z,o)\leq \sum_{z\in S}p_{3}(o,z)\tfrac{1}{D}\leq 6/D^2$. It is not hard to see that $p_4(x,y)\leq p_4(o,o)\leq 6/D^2$ for every $x$ and $y$ in $V$ -- see e.g. \cite[Exercise 6.40]{LyoPer17}. For $n>4$, just observe that  $p_n(x,y)=\sum_{z\in V}p_{n-4}(x,z)p_4(z,y)\leq \tfrac{6}{D^2}\sum_{z\in V}p_{n-4}(x,z)=6/D^2$. This completes the proof.
\end{proof}

%\bibliography{my_bibli}
%\bibliographystyle{abbrv}

\end{document}